\documentclass[12pt]{amsart}

\newtheorem{theorem}{Theorem}[section]

\newtheorem{lemma}[theorem]{Lemma}
\theoremstyle{proposition}

\newtheorem{corollary}[theorem]{Corollary}
\usepackage[top=1 in, bottom=1 in, left=1.3 in, right=1.3 in]{geometry}
\newtheorem{definition}[theorem]{Definition}
\newtheorem{example}[theorem]{Example}

\theoremstyle{remark}

\newtheorem{remark}[theorem]{Remark}

\numberwithin{equation}{section}

\usepackage{times}
\usepackage{enumerate}
\usepackage{graphicx,adjustbox}
\usepackage{hyperref}
\usepackage{amsmath,amssymb}
\usepackage{amscd}
\usepackage{graphicx}
\usepackage[all]{xy}
\usepackage{mathtools}

\usepackage{xcolor}

\title[Projective Closure of Semigroup Algebras]
{Projective Closure of Semigroup Algebras}
\author{
Joydip Saha
\and
Indranath Sengupta
\and
Pranjal Srivastava
}
\date{}
\address{\small \rm  Barasat college, 1 Kalyani Road, Barasat, West Bengal 700126, INDIA.} 
\email{saha.joydip56@gmail.com}

\address{\small \rm  Discipline of Mathematics, IIT Gandhinagar, Palaj, Gandhinagar, 
Gujarat 382355, INDIA.}
\email{indranathsg@iitgn.ac.in}
\thanks{The second author is the corresponding author and thanks SERB for the grant CRG/2022/007047.}

\address{\small \rm  Department of Mathematics, IISER Bhopal, INDIA.}
\email{pranjal@iiserb.ac.in}
\thanks{The third author thanks IISER B for post-doc fellowship at IISER Bhopal.}
\date{}

\subjclass[2020]{Primary 13H10, 13P10, 20M25.}

\keywords{Affine Semigroups, Gr\"{o}bner bases, Associated graded rings, Betti numbers, 
Cohen-Macaulay.}

\begin{document}

\begin{abstract}
This paper investigates the projective closure of simplicial affine semigroups in $\mathbb{N}^{d}$, $d \geq 2$. 
We present a characterization of the Cohen-Macaulay property for the projective closure of these semigroups 
using Gr\"{o}bner bases. Additionally, we establish a criterion, based on Gr\"{o}bner bases, 
for determining the Buchsbaum property of non-Cohen-Macaulay projective closures of numerical 
semigroup rings. Lastly, we introduce the concept of $k$-lifting for simplicial affine semigroups 
in $\mathbb{N}^d$, and investigate its relationship with the original simplicial affine semigroup.
\end{abstract}

\maketitle

\section{Introduction}

Let $\Gamma$ be an affine semigroup, fully embedded in $\mathbb{N}^{d}$. 
Let $\Gamma$ be minimally generated by the set $\{\mathbf{m}_1,\dots,\mathbf{m}_{d+r}\}$.
The semigroup algebra $\mathbb{K}[\Gamma]$ over a field $\mathbb{K}$ is generated by the 
monomials $\mathbf{x^{m}}$, where $\mathbf{m} \in \Gamma$, with maximal ideal 
$\mathfrak{m}=(\mathbf{x^{m_{1}}},\dots,\mathbf{x^{m_{d+r}}})$. 
Let $I(\Gamma)$ denote the defining ideal of $\mathbb{K}[\Gamma]$, which is the 
kernel of the $\mathbb{K}$=algebra homomorphism  
$\phi:A=\mathbb{K}[z_{1},\dots,z_{d+r}] \rightarrow \mathbb{K}[\mathbf{s}]$, 
such that $\phi(z_{i})=\mathbf{s^{m_{i}}}$, $i=1,\dots,d+r$ \, and \, $\mathbf{s}=s_{1}\dots s_{d}$.
Let us write $\mathbb{K}[\Gamma]\cong A/I(\Gamma)$. The defining ideal $I(\Gamma)$ 
is a binomial prime ideal.
\medskip

Consider a partial relation order on $\mathbb{N}^{d}$:
$\mathbf{a}=(a_{1},\dots,a_{d}) \leq \mathbf{b}=(b_1,\dots,b_d)$ if and only if $a_i \leq b_i, \forall \, i$. 
For indeterminate $\mathbf{s}=s_{1}\dots s_{d},\, \mathbf{t}=t_{1}\dots t_{d}$ ,
consider a map $\phi^{h}:\mathbb{K}[z_{0},z_{1},\dots,z_{d+r}]\rightarrow \mathbb{K}[\mathbf{s,t}] $ defined by $\phi^{h}(z_{i})=\mathbf{s}^{\mathbf{m}_{d+r}-\mathbf{m}_{i}} \mathbf{t}^{m_{i}} $ for $1 \leq i \leq d+r$ and $\phi^{h}(z_{0})= \mathbf{s^{m_{d+r}}}$. 
The image of the map $\phi^{h}$ is the subalgebra $\mathbb{K}[\mathcal{A}]$ 
of $\mathbb{K}[\mathbf{s,t}]$ generated by the monomials whose exponents are the 
generators of the affine semigroup 
$$\Gamma^h=
\langle\{
(\mathbf{m_{d+r}},\mathbf{0}), (\mathbf{m_{d+r}-m_{1}},\mathbf{m}_1), \dots , (\mathbf{m_{d+r}-m_{d+r-1}},\mathbf{m}_{d+r-1}),(\mathbf{0}, \mathbf{m}_{d+r}) \}\rangle.$$
We call the subalgebra $\mathbb{K}[\Gamma^h]$ the projective closure of the affine semigroup $\Gamma$. Its vanishing ideal $I(\Gamma^h) = \ker(\phi^{h})$.
\medskip

The projective closure of the affine semigroup $\Gamma$ is Cohen-Macaulay if the  
ideal $I(\Gamma^h)$ is a Cohen-Macaulay ideal if and only if  $\mathbf{s^{m_{d+r}-m_{1}}t^{m_{1}}},\dots,\mathbf{s^{m_{d+r}-m_{d}}t^{m_{d}}}, \mathbf{s^{m_{d+r}}} $ is a regular sequence in $\mathbb{K}[\mathcal{A}]$ 
(see \cite[Theorem 2.6]{Goto}).
\medskip

Affine semigroups provide a natural extension of numerical semigroups and play a crucial 
role in defining the projective closure of numerical semigroup rings. Gr\"{o}bner basis 
theory has proven to be a valuable technique for studying properties related to numerical semigroups and their projective closures, including the Cohen-Macaulayness of their 
associated graded rings and the algorithm for finding the Frobenius number of 
numerical semigroups \cite{Arslan, Morales}. Subsequently, Herzog and Stamate 
\cite{Herzog-Stamate} established Gr\"{o}bner basis criteria for determining the 
Cohen-Macaulay property of projective closures of affine monomial curves.
\medskip

In this paper, we investigate the projective closure of simplicial affine semigroups 
in $\mathbb{N}^d$, where $d\geq 2$. We provide a criterion for characterizing the 
Cohen-Macaulayness of the projective closure of simplicial affine semigroups using the 
Gr\"{o}bner basis of the defining ideal $I(\Gamma^h)$.
\medskip

Furthermore, we explore the Buchsbaum property of the corresponding projective closures 
in cases where the projective closure of the numerical semigroup ring is non-Cohen-Macaulay. 
Kamoi \cite{Kamoi} examined the Buchsbaum property of simplicial affine semigroups 
by utilizing the Gr\"{o}bner basis of the corresponding defining ideal. In Section 3.2, 
we present a characterization of the Buchsbaum property for non-Cohen-Macaulay 
projective closures using the Gr\"{o}bner basis of the corresponding defining ideal.
\medskip

In Section \ref{LAS}, we study the lifting of simplicial affine semigroups 
in $\mathbb{N}^d$. This construction is motivated by \c{S}ahin's work 
\cite{Sahin} on $k$-lifting of monomial curves for a given $k\in \mathbb{N}$ 
and also inspired by the shifting of simplicial affine semigroups. For $d=1$, 
Herzog and Stamate investigated the Cohen-Macaulayness of the associated graded 
ring of numerical semigroup rings. Similarly, \c{S}ahin demonstrated that if 
the associated graded ring of the original numerical semigroup ring is Cohen-Macaulay, 
then the associated graded ring of its $k$-lifting is also Cohen-Macaulay. 
We establish that the same result holds for the lifting of simplicial affine 
semigroups when $d\geq 2$. If $\Gamma$ is a simplicial affine semigroup and 
$\Gamma_k$ represents the corresponding $k$-lifting of $\Gamma$, the Betti 
numbers of both the rings are identical, written as $\beta_i(\mathbb{K}[\Gamma])=\beta_i(\mathbb{K}[\Gamma_k])$, where $\beta_i(\mathbb{K}[\Gamma])$ 
and $\beta_i(\mathbb{K}[\Gamma_k])$ denote the $i^{th}$ Betti numbers of 
$\mathbb{K}[\Gamma]$ and $\mathbb{K}[\Gamma_k]$ respectively. 

\section{Preliminaries}

\begin{definition}{\rm 
		Let $\Gamma$ be an affine semigroup in $\mathbb{N}^{r}$, minimally generated by $\mathbf{a}_{1},\dots,\mathbf{a}_{n+r}$. The \textit{rational polyhedral cone} generated by $\Gamma$ is defined as
		$$
		\mathrm{cone}(\Gamma)=\big\{\sum_{i=1}^{n+r}\alpha_{i}\mathbf{a}_{i}: \alpha_{i} \in \mathbb{R}_{\geq 0}, \,i=1,\dots,r+n\big\}.
		$$
		The \textit{dimension} of $\Gamma$ is defined as the dimension of the subspace generated by $\mathrm{cone}(\Gamma)$.
	}
\end{definition} 

The $\mathrm{cone}(\Gamma)$ is the intersection of finitely many closed linear half-spaces in $\mathbb{R}^{r}$, 
each of whose bounding hyperplanes contains the origin. These half-spaces are called \textit{support hyperplanes}. 

\begin{definition}\label{Extremal rays} {\rm 
		Suppose $\Gamma$ is an affine semigroup of dimension $r$ in $\mathbb{N}^{r}$. If $r=1$, $\mathrm{cone}(\Gamma)=\mathbb{R}_{\geq 0}$.
		If $r=2$, the support hyperplanes are one-dimensional vector spaces, which are called the 
		\textit{extremal rays} of $\mathrm{cone}(\Gamma)$. If $r >2$, intersection of any two
		adjacent support hyperplanes is a one-dimensional vector space, called an extremal ray 
		of $\mathrm{cone}(\Gamma)$. An element of $\Gamma$ is called an extremal ray of $\Gamma$ 
		if it is the smallest non-zero vector of $\Gamma$ in an extremal ray of $\mathrm{cone}(\Gamma)$. 
	}
\end{definition}

\begin{definition}\label{Extremal}{\rm 
		An affine semigroup $\Gamma$ in $\mathbb{N}^{r}$, is said to be \textit{simplicial}   
		if the $\mathrm{cone}(\Gamma)$ has exactly $r$ extremal rays, i.e., if there exist a set with $r$ elements, say 
		$\{\mathbf{a}_{1},\dots,\mathbf{a}_{r}\} \subset \{\mathbf{a}_{1},\dots,\mathbf{a}_{r},\mathbf{a}_{r+1},\dots,\mathbf{a}_{r+n}\}$, such that they 
		are linearly independent over $\mathbb{Q}$ and $\Gamma \subset \sum\limits_{i=1}^{r}\mathbb{Q}_{\geq 0}\mathbf{a}_{i}$.
	}
\end{definition}

\begin{definition}\emph{
		The \textit{Ap\'{e}ry set} of $\Gamma$ with respect to an element $\mathbf{b} \in \Gamma$ is defined as $\{\mathbf{a} \in \Gamma: \mathbf{a}-\mathbf{b} \notin \Gamma\}$. Let $E=\{\mathbf{a}_{1},\dots,\mathbf{a}_{r}\}$ be the set of extremal rays of $\Gamma$, then the Ap\'{e}ry set of $\Gamma$ with respect to the set $E$ is
		\[
		\mathrm{Ap}(\Gamma,E)=\{a \in \Gamma \mid \mathbf{a}-\mathbf{a}_{i} \notin \Gamma, \, \forall i=1,\dots,r\}=\cap_{i=1}^{r}\mathrm{Ap}(\Gamma,\mathbf{a}_{i}).
		\]}
\end{definition}

For an affine semigroup $\Gamma$, consider the natural partial ordering 
$\prec_{\Gamma}$ on $\mathbb{N}$, such that for all elements 
$x,y \in \mathbb{N}^{r}$, $x \prec_{\Gamma} y \,\,\, \text{if} \,\,\, y-x \in \Gamma$. 

\begin{definition}[\cite{Jafari-Type}]{\rm
		Let $\mathbf{b} \in \mathrm{max}_{\prec_{\Gamma}} \mathrm{Ap}(\Gamma,E)$; the element 
		$\mathbf{b}-\sum_{i=1}^{r}\mathbf{a}_{i} $ is called a \emph{quasi-Frobenius} element of $\Gamma$. 
		The set of all quasi-Frobenius elements of $\Gamma$ is denoted by $\mathrm{QF}(\Gamma)$ 
		and its cardinality is said to be the \emph{type} of $\Gamma$, denoted by $\mathrm{type}(\Gamma)$. 
	}
\end{definition}

\begin{remark}[{\cite[Proposition 3.3]{Jafari-Type}}]
	If $\mathbb{K}[\Gamma]$ is arithmetically Cohen-Macaulay, then the last Betti 
	number of $\mathbb{K}[\Gamma]$ is called the Cohen-Macaulay type of 
	$\mathbb{K}[\Gamma]$, written as $\mathrm{CMtype}(\mathbb{K}[\Gamma])$. 
	Moreover, $\mathrm{type}(\Gamma)=\mathrm{CMtype}(\mathbb{K}[\Gamma])$.
\end{remark}

\begin{theorem}
	The defining ideal $I(\Gamma^h)$ of $\mathbb{K}[\Gamma^h]$ is given by the homogenization of $f$ with respect to variable $z_0$, where $f \in I(\Gamma)$, i.e.,
	$I(\Gamma^h)=\{f^h: f \in I(\Gamma)\}$.
\end{theorem}

\begin{proof} Let $f \in \mathrm{ker}({\phi^h})$, then $\phi^h(f(z_0,\dots,z_{d+r}))=0$, 
which implies that 
$$f(\mathbf{s}^{\mathbf{m}_{d+r}},\mathbf{s}^{\mathbf{m}_{d+r}-\mathbf{m}_{i}} \mathbf{t}^{m_{i}},\dots,\mathbf{t}^{\mathbf{m}_{d+r}})=0.$$ 
Therefore, $  
	(\mathbf{s}^{\mathbf{m}_{d+r})^{\mathrm{deg}\, f}}f(1,\frac{\mathbf{t}^{\mathbf{m}_1}}{\mathbf{s}^{\mathbf{m}_1}},\dots,\frac{\mathbf{t}^{\mathbf{m}_{d+r}}}{\mathbf{t}^{\mathbf{m}_{d+r}}})=0$, hence
	$ f(1,\frac{\mathbf{t}^{\mathbf{m}_1}}{\mathbf{s}^{\mathbf{m}_1}},\dots,\frac{\mathbf{t}^{\mathbf{m}_{d+r}}}{\mathbf{t}^{\mathbf{m}_{d+r}}})=0$. 
	 Now we have $ \mathbb{K}[\frac{\mathbf{t}}{\mathbf{s}}]\cong \mathbb{K}[\mathbf{s}]$, 
	 thus 
	 $f(1,\mathbf{s}^{\mathbf{m}_1},\dots,\mathbf{s}^{\mathbf{m}_{d+r}})=0$ and $ f(1,z_1,\dots,z_{d+r}) \in \mathrm{ker}(\phi)$. Which implies 
	   $z_{0}^{\mathrm{deg}\,f} f(1,\frac{z_1}{z_0},\dots,\frac{z_{d+r}}{z_0}) \in 
	   \mathrm{ker}(\phi^h)$.
\smallskip
	
	Conversely, if $f^h \in (\mathrm{ker}(\phi))^h$, where $f \in \mathrm{ker}(\phi)$, 
	then $f(\mathbf{s}^{\mathbf{m}_1},\dots,\mathbf{s}^{\mathbf{m}_{d+r}})=0$, but $\mathbb{K}[\frac{\mathbf{t}}{\mathbf{s}}]\cong \mathbb{K}[\mathbf{s}]$. Therefore, $f(\frac{\mathbf{t}^{\mathbf{m}_1}}{\mathbf{s}^{\mathbf{m}_1}},\dots,\frac{\mathbf{t}^{\mathbf{m}_{d+r}}}{\mathbf{t}^{\mathbf{m}_{d+r}}})=0$, which implies that $$(\mathbf{s}^{\mathbf{m}_{d+r}})^{\mathrm{deg}\,f}f(\mathbf{s}^{\mathbf{m}_{d+r}},\mathbf{s}^{\mathbf{m}_{d+r}-\mathbf{m}_{i}} \mathbf{t}^{m_{i}},\dots,\mathbf{t}^{\mathbf{m}_{d+r}})=0.$$
	Hence, $f^h \in \mathrm{ker}(\phi^h)$. 
\end{proof}

\section{Cohen-Macaulayness of Projective closure of simplicial affine semigroups}
In this section, we study the Cohen-Macaulay characterization of the projective closure of  simplicial affine semigroups in $\mathbb{N}^d$ in terms of a Gr\"{o}bner basis of 
the defining ideal of corresponding semigroups.

\begin{theorem}\label{Extremal rays}
Let $E_{\Gamma}=\{\mathbf{m_{1}},\dots,\mathbf{m_{d}}\}$ be a set of extremal rays of $\Gamma$. Then the set $$E_{\Gamma^h}=\{(\mathbf{m}_{d+r}-\mathbf{m}_{i},\mathbf{m}_{i}),(\mathbf{m}_{d+r},\mathbf{0}) \vert 1 \leq i \leq d\}$$ is a set of extremal rays of $\Gamma^h$.
\end{theorem}

\begin{proof}
Let $\mathbf{m}_{j} \in \Gamma$ for $j>0$, then there exists $\alpha_{1},\dots,\alpha_{d} \in \mathbb{Q}$ such that $\mathbf{m}_{j}= \sum_{i=1}^{d}\alpha_{i}\mathbf{m}_{i}$. For  $(\mathbf{m}_{d+r}-\mathbf{m}_{j},\mathbf{m}_{j}) \in \overline{\Gamma}$, we have 
\begin{align*}
	(\mathbf{m}_{d+r}-\mathbf{m}_{j},\mathbf{m}_{j})=&\sum_{i=2}^{d}\big(\alpha_{i}(\mathbf{m}_{d+r}-\mathbf{m}_{i},\mathbf{m}_{i})-\alpha_{2}(\mathbf{m}_{d+r},0)\big)\\
	&+\big(\alpha_{1}(\mathbf{m}_{d+r}-\mathbf{m}_{1},\mathbf{m}_{1})-(\alpha_{1}-1)(\mathbf{m}_{d+r},0)\big).
\end{align*}
It is clear that $\{(\mathbf{m}_{d+r}-\mathbf{m}_{i},\mathbf{m}_{i}),(\mathbf{m}_{d+r},\mathbf{0}) \vert 1 \leq i \leq d\}$ is linearly independent over $\mathbb{Q}$. Hence, $\{(\mathbf{m}_{d+r}-\mathbf{m}_{i},\mathbf{m}_{i}),(\mathbf{m}_{d+r},\mathbf{0}) \vert 1 \leq i \leq d\}$ is a set of extremal rays of $\overline{\Gamma}$.
\end{proof}

Let $\prec$ denotes the degree reverse lexicographic ordering on $A=\mathbb{K}[z_{1},\dots,z_{d+r}]$ induced by $z_{1}\prec \dots \prec z_{d+r}$ and $\prec_{0}$ the induced reverse lexicographic order on $A[z_{0}]$, where $z_0 \prec z_1$. The following theorem is one of the main results of this paper. For any monomial ideal $I$, we let $G(I)$ denote the unique minimal set of monomial generators for $I$.

\begin{theorem}
Let $\Gamma$ be a simplicial affine semigroup in $\mathbb{N}^d, d\geq 2$ with set of extremal rays $E=\{\mathbf{m}_{1},\dots,\mathbf{m}_{d}\}$. Then, the followings are equivalent.
\begin{enumerate}[(a)]

\item $\mathbb{K}[\Gamma]$ is arithmetically Cohen-Macaulay.

\item $\mathbb{K}[\Gamma^h]$ is arithmetically Cohen-Macaulay.

\item $z_0,z_1,\dots,z_d$ do not divide any element of $G(\mathrm{in}_{\prec_0}(I(\Gamma^h))$.

\item $z_1,\dots,z_d$ do not divide any element of $G(\mathrm{in}_{\prec}(I(\Gamma))$.
\end{enumerate}
\end{theorem}

\begin{proof} (a) if and only if (d) follows from \cite[Corollary 4.7]{Ojeda-Short} and 
(c) if and only if (d), follows from the fact $G(\mathrm{in}_{\prec}(I(\Gamma))=G(\mathrm{in}_{\prec_0}(I(\Gamma^h)))$ (See \cite[Lemma 2.1]{Herzog-Stamate}).
\medskip
  
\noindent{$\mathbf{(a) \Rightarrow (b)}$.} 
Suppose $\mathbb{K}[\Gamma]$ is Cohen-Macaulay. By \cite[Corollary 4.7]{Ojeda-Short}, 
$z_{1},\dots,z_{d}$ do not divide initial ideal $\mathrm{in}_{\prec}(I(\Gamma))$. 
Also,  $z_{1},\dots,z_{d}$ form a regular sequence in $A/I(\Gamma)$, hence 
$z_{1},\dots,z_{d}$ form a regular sequence in $A/\mathrm{in}_{\prec}(I(\Gamma)$, 
by \cite[Theorem 15.13]{Eisenbud}. Since $G(\mathrm{in}_{\prec}(I(\Gamma))=G(\mathrm{in}_{\prec_0}(I(\Gamma^h)))$, we have  $z_0$ is regular element in  $A[z_0]/((z_1,\dots,z_d)+\mathrm{in}_{\prec}(I(\Gamma)))=A[z_0]/((z_1,\dots,z_d)+\mathrm{in}_{\prec}(I(\Gamma^h)))$. Hence, $z_{1},\dots,z_{d},z_{0}$ form a regular sequence in $A[z_0]/\mathrm{in}_{\prec}(I(\Gamma^h))$. Again by \cite[Theorem 15.13]{Eisenbud}, $z_{1},\dots,z_{d},z_{0}$ 
is a regular sequence in $A[z_0]/I(\Gamma^h))$. 
Hence, $\mathbf{s^{m_{d+r}-m_{1}}t^{m_{1}}},\dots,\mathbf{s^{m_{d+r}-m_{d}}t^{m_{d}}}, \mathbf{s^{m_{d+r}}} $ is a regular sequence in $\mathbb{K}[\Gamma^h]$. Therefore $\mathbb{K}[\Gamma^h]$ is Cohen-Macaulay.
\medskip

\noindent{$\mathbf{(b) \Rightarrow (c)}$.} 
 By Lemma \ref{Extremal rays}, $\{(\mathbf{m}_{d+r}-\mathbf{m}_{i},\mathbf{m}_{i}),(\mathbf{m}_{d+r},\mathbf{0}) \vert 1 \leq i \leq d\}$ is a set of extremal rays of $\Gamma^h$. Again  by \cite[Corollary 4.7]{Ojeda-Short} and due to Lemma \ref{Extremal rays}, 
$\mathbb{K}[\Gamma^h]$ is Cohen-Macaulay if and only if $(\mathbf{s^{m_{d+r}-m_{1}}t^{m_{1}}},\dots,\mathbf{s^{m_{d+r}-m_{d}}t^{m_{d}}}, \mathbf{s^{m_{d+r}}}) $ is a regular sequence in $\mathbb{K}[\Gamma^h]$. This is equivalent to the fact that 
$z_{1},\dots,z_{d},z_{0}$ is a regular sequence in $A[z_0]/\mathrm{in}_{\prec}(I(\Gamma^h))$. Therefore, $z_0,z_1,\dots,z_d$ do not divide $G(\mathrm{in}_{\prec_0}(I(\Gamma^h))$.
\end{proof}

Let $ Q_{I(\Gamma)}=\{(\alpha_{d+1},\dots,\alpha_{d+r}) \in \mathbb{N}^{r} \mid z_{d+1}^{\alpha_{d+1}}\dots z_{d+r}^{\alpha_{d+r}} \notin \mathrm{in}_{\prec}I(\Gamma)\}$ and 
$Q_{I(\Gamma^{h})}=\{(\alpha_{d+1},\dots,\alpha_{d+r}) \in \mathbb{N}^{r} \mid z_{d+1}^{\alpha_{1}}\dots z_{d+r}^{\alpha_{d+r}} \notin \mathrm{in}_{\prec_{0}}I(\Gamma^{h})\}$. 
It follows from Proposition 3.3 in \cite{Ojeda-Short}, there is a bijection between $Q_{I(\Gamma^{h})}$ and $\mathrm{Ap}(\Gamma^h,E_\Gamma^h)$ given by 
$(\alpha_{d+1},\dots,\alpha_{d+r}) \mapsto \sum_{i=d+1}^{d+r}\alpha_{i}(\mathbf{m}_{d+r}-\mathbf{m}_{i},\mathbf{m}_{i})$. 
Similarly, $(\alpha_{d+1},\dots,\alpha_{d+r}) \mapsto \sum_{i=d+1}^{d+r}\alpha_{i}\mathbf{m}_{i}$ is a bijection between $Q_{I(\Gamma)}$ and $\mathrm{Ap}(\Gamma,E_\Gamma)$.

\begin{theorem}
Let $\mathbb{K}[\Gamma^{h}]$ be arithmetically Cohen-Macaulay.
The Ap\'{e}ry set of $\Gamma^{h}$ with respect to $E_{\Gamma^{h}}$ is given by $\mathrm{Ap}(\Gamma^{h},E_{\Gamma^h})=$
$$\left\{\sum_{i=d+1}^{d+r}\alpha_{i}(\mathbf{m}_{d+r}-\mathbf{m}_{i},\mathbf{m}_{i}) :  \sum_{i=d+1}^{d+r}\alpha_{i}\mathbf{m}_{i} \in \mathrm{Ap}(\Gamma,E_\Gamma),\, \mbox{for}\, \mbox{some}\, (\alpha_{d+1},\dots,\alpha_{d+r})\in \mathbb{N}^r\right\}.$$
\end{theorem}

\begin{proof}
Let  $\mathbf{\gamma^h}\in \mathrm{Ap}(\Gamma^{h},E_{\Gamma^{h}}) \setminus \{\mathbf{0}\}$. 
We have $ \mathbf{\gamma^h}=\sum_{i=d+1}^{d+r}\alpha_{i}(\mathbf{m}_{d+r}-\mathbf{m}_{i},\mathbf{m}_{i})$, for some $ (\alpha_{d+1},\dots,\alpha_{d+r})\in \mathbb{N}^r$. 
By the bijection defined above, $z_{d+1}^{\alpha_{d+1}}\dots z_{d+r}^{\alpha_{d+r}} \notin \mathrm{in}_{\prec_0}(I(\Gamma^h))$. However, 
$G(\mathrm{in}_{\prec_0}(I(\Gamma^h)))=G(\mathrm{in}_{\prec}(I(\Gamma)) \text{ implies } z_{d+1}^{\alpha_{d+1}}\dots z_{d+r}^{\alpha_{d+r}} \notin \mathrm{in}_{\prec}(I(\Gamma))$.
Hence, by the same bijective map,
$ \sum_{i=d+1}^{d+r}\alpha_{i}\mathbf{m}_{i} \in \mathrm{Ap}(\Gamma,E_\Gamma).
$
\end{proof}

For an affine semigroup $\Gamma$, we consider the partial natural ordering $\preceq_{\Gamma}$ on $\mathbb{N}^{d}$, defined as $\mathbf{x}\preceq_{\Gamma} \mathbf{y}$ if $\mathbf{y}-\mathbf{x} \in \Gamma$.

\begin{theorem}
Let $\mathbb{K}[\Gamma^h]$ be Cohen-Macaulay. Then, the Cohen-Macaulay type of $\mathbb{K}[\Gamma^h]$ is $\mathrm{CMtype}(\mathbb{K}[\Gamma^h])=|\{\mathrm{max}_{\preceq_{\Gamma^h}}(\mathrm{Ap}(\Gamma^{h},E_{\Gamma^{h}}))\}|$.
\end{theorem}

\begin{proof}
	Due to Lemma \ref{Extremal rays}, $\{\mathbf{x}^{\mathbf{m}}:\mathbf{m} \in E_{\Gamma^h} \}$ provides a monomial system of parameter for $\mathbb{K}[\Gamma^h]$. As $\mathbb{K}[\Gamma^h]$ is Cohen-Macaulay, $\{\mathbf{x}^{\mathbf{m}}:\mathbf{m} \in E_{\Gamma^h} \}$ forms a regular sequence in $\mathbb{K}[\Gamma^h]$. Now the proof is similar as it is in 
	\cite[Proposition 3.3]{Jafari-Type}.
\end{proof}

\section{Buchsbaum Criterion for non-Cohen Macaulay Projective Closure of Numerical Semigroup rings}
Let $\Gamma$ be a numerical semigroup minimally generated by $\{n_1,\dots,n_{r}\}\subset \mathbb{N}$, i.e $\Gamma=\langle n_1,\dots,n_r \rangle$. 
Consider a map $\phi:\mathbb{K}[x_0,\dots,x_r] \rightarrow \mathbb{K}[u^{n_r},u^{n_{r}-n_{1}}v^{n_{1}},\dots,v^{n_{r}}]$ defined by $\phi(x_{i})=u^{n_r-n_i}v^{n_i}$. Let $I_{\Gamma^h}$ be a defining ideal of $\mathbb{K}[\Gamma^h]$.
Assume that the projective closure $\mathbb{K}[\Gamma^h]$ of $\mathbb{K}[\Gamma]$ is not Cohen-Macaulay. 
\medskip

Let $T$ be a set of tuple $(\alpha,\beta)\in \mathbb{N}^2$ such that 
$u^\alpha v^\beta \in \mathbb{K}[\Gamma^h]$, and $e_{i}=(n_{r},n_{r}-n_{i}), 0\leq i \leq r$.
\medskip 
 
 Given two subset $A, B \subset \mathbb{N}^2$, define $A\pm B=\{a\pm b : a\in A, b\in B\}$. Now, define $T^{\star}=(T-e_0)\cap (T-e_1)\cap \dots \cup (T-e_{r})$. Note that $T^{\star}$ is an additive semigroup in $\mathbb{N}^2$ and finitely generated. Suppose affine semigroup $T^{\star}$ is minimally generated by $\{(n_r,0),(\alpha_1,\beta_1),\dots,(\alpha_s,\beta_s),(0,n_r)\}$. 
It is clear that that $T^{\star}$ is simplicial affine semigroup in $\mathbb{N}^{2}$. 
\medskip

Consider a map, $\phi_{T^{\star}}:\mathbb{K}[x_0,\dots,x_{s+1}] \rightarrow \mathbb{K}[T^{\star}]$, defined by $\phi_{T^{\star}}(x_{0})=u^{n_r}, \phi_{T^{\star}}(x_{j})=u^{\alpha_j}v^{\beta_j},\phi_{T^{\star}}(x_{s+1})=v^{n_r},\, 1 \leq j \leq s$. 
Let $I_{T^{\star}}$ be the defining ideal of $\mathbb{K}[T^{\star}]$. 
Let $G$ be a reduced Gr\"{o}bner basis of $\mathbb{K}[T^{\star}]$ with 
respect to degree reverse lexicographic ordering $>$ induced by 
$x_1>\dots>x_{s}>x_{s+1}>x_{0}$. By \cite[Theorem 5]{Buchsbaum-Pedro}, 
$\mathbb{K}[\Gamma^h]$ is Buchsbaum if and only if $\mathbb{K}[T^{\star}]$ is 
Cohen-Macaulay.  As $T^{\star}$ is a simplicial affine semigroup in 
$\mathbb{N}^{2}$, from \cite[Corollary 4.7]{Ojeda-Short}, 
$\mathbb{K}[T^{\star}]$ is Cohen-Macaulay if and only if $x_{0},x_{s+1}$ 
do not divide the leading monomial of any element of $G$. Thus, we deduce 
the following theorem. 
 
\begin{theorem}
$\mathbb{K}[\Gamma^h]$ is Buchsbaum if and only if $x_0,x_{s+1}$ do not divide the leading monomial of any element of $G$. 
\end{theorem}

\begin{example}
Let $R=\mathbb{K}[u^4,u^3v,uv^3,v^4]$. Here $ T=\langle (0,4),(3,1),(1,3),(0,4)\rangle$ and $T^{\star}$ is minimally generated by the set $\langle (0,4),(1,3),(2,2),(3,1),(4,0) \rangle$ with extremal rays $\{(0,4),(4,0)\}$. Note that $G=\{x_3^2-x_4x_0,x_2x_3-x_1x_0,x_1x_3-x_2x_4,x_2^2-x_3x_0,x_1x_2-x_4x_0,x_1^2-x_3x_4\}$ is the reduced Gr\"{o}bner basis of $I_{T_{\star}}$ with respect to degree reverse lexicographic ordering $>$ induced by $x_1>x_2>x_3>x_4>x_0$. We can see that $x_0,x_4$ do not divide the leading monomial of any element of $G$, hence $R$ is a Buchsbaum ring. 
\end{example}

\section{Lifting of affine semigroups}\label{LAS}
Let $\Gamma$ be a simplicial affine semigroup minimally generated by $\{\mathbf{a}_1,\dots,\mathbf{a}_{d+r}\} $ with the set of extremal rays $\{\mathbf{a}_1,\dots,\mathbf{a}_{d}\}$. 
The \textit{$k$ lifting of $\Gamma$} is defined as the affine semigroup $\Gamma_{k}$ minimally generated by $\{\mathbf{a}_1,\dots,\mathbf{a}_d,k\mathbf{a}_{d+1},\dots,k\mathbf{a}_{d+r}\}$. Note that $\Gamma_{k}$ is simplicial affine semigroup.
Let $x_{E}^{\alpha}=x_1^{\alpha_1}\dots x_d^{\alpha_d}$, $x_{E_c}^{\alpha'}=x_{d+1}^{\alpha_{d+1}}\dots x_{d+r}^{\alpha_{d+r}}$, where $\alpha=(\alpha_1,\dots,\alpha_d)$,  $\alpha'=(\alpha_{d+1},\dots,\alpha_{d+r})$. 

\begin{theorem}\label{Gen} 
Let $x_{E}^{\alpha}x_{E_c}^{\alpha'}-x_{E}^{\beta}x_{E_c}^{\beta'} \in I_{\Gamma}$, 
then $x_{E}^{k\alpha}x_{E_c}^{\alpha'}-x_{E}^{k\beta}x_{E_c}^{\beta'} \in I_{\Gamma_k}$. Moreover, $\mu(I_{\Gamma})=\mu(I_{\Gamma_k})$.
\end{theorem}

\begin{proof}
As $\sum_{i=1}^{d}\alpha_i \mathbf{a}_i+\sum_{i=d+1}^{d+r}\alpha_i k\mathbf{a}_i -k\sum_{i=1}^{d}\beta_i \mathbf{a}_i+\sum_{i=d+1}^{d+r}\beta_i k\mathbf{a}_i=0$, we have $k\sum_{i=1}^{d}\alpha_i \mathbf{a}_i+\sum_{i=d+1}^{d+r}\alpha_i k\mathbf{a}_i -k\sum_{i=1}^{d}\beta_i \mathbf{a}_i+\sum_{i=d+1}^{d+r}\beta_i k\mathbf{a}_i=0 $. Hence, $x_{E}^{k\alpha}x_{E_c}^{\alpha'}-x_{E}^{k\beta}x_{E_c}^{\beta'} \in I_{\Gamma_k}$.
\medskip

Note that $A=\mathbb{K}[x_1,\dots,x_{d+r}]$ is graded over both $\Gamma$ and  
$\Gamma_k$, via $\mathrm{deg}_{\Gamma}(x_i)=\mathrm{deg}_{\Gamma_k}(x_i)=\mathbf{a}_{i}$, 
for $1\leq i \leq d$, and $\mathrm{deg}_{\Gamma_k}(x_i)=k\mathbf{a}_{i}$, for 
$d+1\leq i \leq d+r$, respectively. Consider a $\mathbb{K}$-algebra homomorphism $f:A \rightarrow A$ 
defined by 
$x_{i} \rightarrow x_{i}^k $ for $1\leq i \leq d$ and  $x_{i} \rightarrow x_{i} $ 
for $d+1\leq i \leq d+r$. Since $f$ is faithfully flat extension, $\mathbb{F} \otimes A$ 
gives the minimal $\Gamma_k$-graded free resolution of $\mathbb{K}[\Gamma_k]$, 
where $\mathbb{F}$ is a minimal $\Gamma$-graded free resolution of $\mathbb{K}[\Gamma]$. 
Hence, Betti numbers of $\mathbb{K}[\Gamma]$ match with the Betti numbers of $\mathbb{K}[\Gamma_k]$, i.e.,  $\beta_i(\mathbb{K}[\Gamma])=\beta_i(\mathbb{K}[\Gamma_k])$. In particular, $\mu(I_{\Gamma})=\mu(I_{\Gamma_k})$.
\end{proof}

\begin{corollary}\label{CM}
If $\mathbb{K}[\Gamma]$ is Cohen-Macaulay (respectively Gorenstein), 
then $\mathbb{K}[\Gamma_k]$ is Cohen-Macaulay (respectively Gorenstein).
\end{corollary}

\begin{remark}
It is clear that we can get a generating set of $I_{\Gamma_k}$ from generating set of $I_{\Gamma}$ by taking $k^{th}$ power of variable corresponding to extremal rays.
\end{remark}

Let $\prec$ denotes the reverse lexicographic ordering on $\mathbb{K}[x_1,\dots,x_{d+r}]$ induced by $x_{1}\prec \dots \prec x_{d+r}$.

\begin{lemma}
    If $\mathbb{K}[\Gamma]$ is Cohen-Macaulay then $|\mathrm{Ap}(\Gamma_k,E)|=|\mathrm{Ap}(\Gamma,E)|$.
\end{lemma}

\begin{proof}
From Theorem \ref{Gen}, it is clear that $\mathrm{dim}_{\mathbb{K}}\frac{\mathbb{K}[x_1,\dots,x_{d+r}]}{I_{\Gamma}+(x_1,\dots,x_d)}=\mathrm{dim}_{\mathbb{K}}\frac{\mathbb{K}[x_1,\dots,x_{d+r}]}{I_{\Gamma_{k}}+(x_1,\dots,x_d)}$. By \cite[Theorem 3.3]{Ojeda-Short}, we have $|\mathrm{Ap}(\Gamma_k,E)|=|\mathrm{Ap}(\Gamma,E)|$. 
\end{proof}

\begin{theorem}
    If $\mathbb{K}[\Gamma]$ is Cohen-Macaulay then $\mathrm{Ap}(\Gamma_k,E)=\{k\mathbf{b} : \mathbf{b} \in \mathrm{Ap}(\Gamma_k,E)\}$.
\end{theorem}

\begin{proof} Let $\mathbf{b}_k \in \mathrm{Ap}(\Gamma_k,E)$, then $\mathbf{b}_k-\mathbf{a}_i \notin \Gamma_k, \forall i=1,\dots,d$. We can write $\mathbf{b}_k=k\mathbf{a}$, where $\mathbf{a} \in \langle \mathbf{a}_{d+1},\dots,\mathbf{a}_{d+r}\rangle$. So, $k\mathbf{a}-\mathbf{a}_i \notin \Gamma_k$ implies that $k\mathbf{a}-k\mathbf{a}_i \notin \Gamma_k$. Hence,  $\mathbf{a}-\mathbf{a}_i \notin \Gamma$, and $\mathbf{a}\in \mathrm{Ap}(\Gamma,E)$. Since, 
    $|\mathrm{Ap}(\Gamma_k,E)|=|\mathrm{Ap}(\Gamma,E)|$, therefore reverse equality follows.
\end{proof}

\begin{definition}
\emph{A semigroup $\Gamma$ is said to be \emph{of homogeneous type} if 
$\beta_{i}(\mathbb{K}[\Gamma])=\beta_{i}(\mathrm{gr}_{\mathfrak{m}}(\mathbb{K}[\Gamma]))$ for all $i \geq 1$}. 
\end{definition}

\begin{theorem}
Let $\mathrm{gr}_{\mathfrak{m}}(\mathbb{K} [\Gamma])$ be Cohen-Macaulay. If 
$\Gamma$ is of homogeneous type then $\Gamma_k$ is of homogeneous type.
\end{theorem}

\begin{proof}
Let $B$ be a binomial in $I_{\Gamma}$ with monomial having no common divisor. Since $\mathrm{gr}_{\mathfrak{m}}(\mathbb{K}[\Gamma]))$ is Cohen-Macaulay, then either $x_1,\dots,x_d$ divides no monomial in $B$ or divides only one of them. In first case, $I_{\Gamma}=I_{\Gamma_k}$, so $B$ lies in $I_{\Gamma_k} $ and define $B_k:=B$. In the latter case, we define $B_k=x_{d+1}^{u_{d+1}}\dots x_{d+r}^{u_{d+r}}-x_{1}^{kv_1}\dots x_{d}^{kv_{d}}x_{d+1}^{v_{d+1}}\dots x_{d+r}^{v_{d+r}}$, hence $\pi_d(I_{\Gamma_k})=\pi_d(I_{\Gamma})$.
\medskip

We consider the map $\pi_{d}:A=\mathbb{K}[x_1,\dots,x_{d},\dots,x_{d+r}] \rightarrow \bar{A}=\mathbb{K}[x_{d+1},\dots,x_{d+r}]$, such that $\pi_{d}(x_{j})=0$, 
$1 \leq j \leq d$ and $\pi_{d}(x_{j})=x_{j}, d+1 \leq j \leq d+r$.
\medskip

Since $\mathrm{gr}_{\mathfrak{m}_{k}}(\mathbb{K} [\Gamma_{k}])$ is Cohen-Macaulay, by Lemma 3 in \cite{Saha-Associated} we have $$\mathrm{gr}_{\bar{\mathfrak{m}}}(\bar{A}/\pi_d(I_{\Gamma_k}))]) \cong \frac{\mathrm{gr}_{\mathfrak{m}}(A/I_{\Gamma_k})}{(x_1,\dots,x_d)\mathrm{gr}_{\mathfrak{m}}(A/I_{\Gamma_k})}.$$
Therefore 
$\beta_i(\mathrm{gr}_{\bar{\mathfrak{m}}}(\bar{A}/\pi_d(I_{\Gamma_k}))])=
\beta_i( \frac{\mathrm{gr}_{\mathfrak{m}}(A/I_{\Gamma_k})}{(x_1,\dots,x_d)\mathrm{gr}_{\mathfrak{m}}(A/I_{\Gamma_k})})=\beta_i(\mathrm{gr}_{\mathfrak{m}}(A/I_{\Gamma_k}))$,  
since $x_1,\dots,x_d$ are non-zero-divisors on $\mathrm{gr}_{\mathfrak{m}}(A/I_{\Gamma_k})$. 
\medskip

Similarly, given that $\mathrm{gr}_{\mathfrak{m}}(\mathbb{K} [\Gamma])$ is Cohen-Macaulay and $\Gamma$ is of homogeneous type, we can write $\beta_i(\mathrm{gr}_{\bar{\mathfrak{m}}}(\bar{A}/\pi_d(I_{\Gamma}))]) = \beta_I(\mathrm{gr}_{\mathfrak{m}}(A/I_{\Gamma}))=\beta_i(A/I_{\Gamma})$. Therefore, 
\begin{align*}
\beta_i(\mathrm{gr}_{\mathfrak{m}}(A/I_{\Gamma_k}))&=\beta_i(\mathrm{gr}_{\bar{\mathfrak{m}}}(\bar{A}/\pi_d(I_{\Gamma_k}))])\\
&=\beta_i(\mathrm{gr}_{\bar{\mathfrak{m}}}(\bar{A}/\pi_d(I_{\Gamma}))]) \\
&= \beta_I(\mathrm{gr}_{\mathfrak{m}}(A/I_{\Gamma}))=
\beta_i(A/I_{\Gamma})=\beta_i(A/I_{\Gamma_k})\, (\text{by Theorem \ref{Gen}}).
\end{align*}
Hence, $\Gamma_k$ is of homogeneous type. 
    
\end{proof}

\section*{Data Availability} This manuscript has no associated data. 

\bibliographystyle{amsalpha}

\end{document}